\documentclass[12pt,reqno]{amsart}

\usepackage{multirow}
\usepackage{hhline}

\usepackage{times}
\usepackage[T1]{fontenc}
\usepackage{mathrsfs}
\usepackage{latexsym}
\usepackage[dvips]{graphics}
\usepackage[titletoc,title]{appendix}
\usepackage{epsfig}

\usepackage{amsmath,amsfonts,amsthm,amssymb,amscd}
\input amssym.def
\input amssym.tex
\usepackage{color}
\usepackage{hyperref}

\usepackage{color}
\usepackage{breakurl}
\usepackage{comment}
\newcommand{\bburl}[1]{\textcolor{blue}{\url{#1}}}

\newcommand{\be}{\begin{equation}}
\newcommand{\ee}{\end{equation}}
\newcommand{\bea}{\begin{eqnarray}}
\newcommand{\eea}{\end{eqnarray}}

\newtheorem{thm}{Theorem}[section]

\newtheorem{lem}[thm]{Lemma}

\newtheorem{rek}[thm]{Remark}

\numberwithin{equation}{section}

\begin{document}

\title{On Generalized Zeckendorf Decompositions and Generalized Golden Strings}

\author{H\`ung Vi\d{\^e}t Chu}
\email{\textcolor{blue}{\href{mailto:hungchu2@illinois.edu}{hungchu2@illinois.edu}}}
\address{Department of Mathematics, University of Illinois at Urbana-Champaign, Urbana, IL 61820, USA}



\thanks{}

\date{\today}
\maketitle

\begin{abstract}
Zeckendorf proved that every positive integer has a unique representation as a sum of non-consecutive Fibonacci numbers. A natural generalization of this theorem is to look at the sequence defined as follows: for $n\ge 2$, let $F_{n,1} = F_{n,2} = \cdots = F_{n,n} = 1$ and $F_{n, m+1} = F_{n, m} + F_{n, m+1-n}$ for all $m\ge n$. It is known that every positive integer has a unique representation as a sum of $F_{n,m}$'s where the indexes of summands are at least $n$ apart. We call this the $n$-decomposition. Griffiths showed an interesting relationship between the Zeckendorf decomposition and the golden string. In this paper, we continue the work to show a relationship between the $n$-decomposition and the generalized golden string. 
\end{abstract}

\section{Introduction} We define the Fibonacci sequence to be $F_1 = 1$, $F_2 = 1$, and $F_m = F_{m-1}+F_{m-2}$ for $m\ge 3$. The Fibonacci numbers have fascinated mathematicians for centuries with many interesting properties. 
A beautiful theorem of Zeckendorf \cite{Z} states that every positive integer $m$ can be uniquely written as a sum of non-consecutive Fibonacci numbers. This gives the so-called Zeckendorf decomposition of $m$. A more formal statement of Zeckendorf's theorem is as follows:
\begin{thm}\label{p1}
For any $m\in\mathbb{N}$, there exists a unique increasing sequence of positive integers $(c_1, c_2, \ldots, c_k)$ such that $c_1\ge 2$, $c_i\ge c_{i-1}+2$ for $i = 2, 3, \ldots, k$, and $m = \sum_{i=1}^kF_{c_i}$. 
\end{thm}
Much work has been done to understand the structure of Zeckendorf decompositions and their applications (see \cite{BDEMMTW1, BDEMMTW2, B, CSH, MG1, MG2, HS, L}) and to generalize them (see \cite{DDKMMV, DFFHMPP, Ho, K, ML}). Before stating our main results, we mention several related results from the literature. Given $n\in \mathbb{N}_{\ge 2}$, we define the sequence $F_{n,1} = \cdots = F_{n,n} = 1$ and $F_{n,m+1} = F_{n,m} + F_{n,m+1-n}$ for all $m\ge n$. The following theorem, which is a generalization of Theorem \ref{p1}, follows immediately from the proof of \cite[Theorem 1.3]{DDKMMV}.
\begin{thm}\label{p2}
For any $m\in\mathbb{N}$, there exists a unique increasing sequence of positive integers $(c_1, c_2, \ldots, c_k)$ such that $c_1\ge n$, $c_i\ge c_{i-1}+n$ for $i = 2, 3, \ldots, k$, and $m = \sum_{i=1}^kF_{n,c_i}$. 
\end{thm}
For conciseness, we call the decomposition in Theorem \ref{p2} the $n$-decomposition. In \cite{MG2}, 
Griffiths made a connection between the golden string and the Zeckendorf decomposition. In particular, the golden string $S_{\infty} = a_2a_1a_2a_2a_1a_2a_1a_2a_2a_1a_2a_2a_1\ldots$ is defined to be the infinite string of $a_1$ and $a_2$ constructed recursively as follows. Let $S_1 = a_1$ and $S_2 = a_2$, and for $m\ge 3$, $S_m$ is defined to be the concatenation of the strings $S_{m-1}$ and $S_{m-2}$, which we denote by $S_{m-1}\circ S_{m-2}$. Thus, 
\begin{align*}
    S_3 &\ =\ S_2\circ S_1 \ =\ a_2\circ a_1 \ =\ a_2a_1,\\
    S_4 &\ =\ S_3\circ S_2 \ =\ a_2a_1\circ a_2\ =\ a_2a_1a_2,
\end{align*}
and so on. We generalize the golden string in an obvious way. Given $n\in\mathbb{N}_{\ge 2}$, we let 
\begin{align*}
    S_1 &\ =\ a_1, \ldots, S_n \ =\ a_n,\\
    S_{m} &\ =\ S_{m-1} \circ S_{m-n} \mbox{ for }m\ge n+1,
\end{align*}
We call $S_\infty$ obtained from the recursive process the $n$-string. For example, when $n = 2$, we have the golden string; when $n = 3$, we have the $3$-string: $$a_3a_1a_2a_3a_3a_1a_3a_1a_2a_3a_1a_2a_3a_3\ldots .$$
Lemmas 3.1 and 3.2 in \cite{MG2} show that the Zeckendorf decomposition ($2$-decomposition) is linked to the golden string ($2$-string). This fact might lead us to suspect that in general, the $n$-decomposition is linked to the $n$-string. Our first three theorems show that the suspicion is indeed well-founded. 

\begin{thm}\label{m1}
Let $n\ge 3$ and $m\ge 1$. The following items hold.
\begin{itemize}
    \item [(1)] $S_m$ contains $F_{n,m}$ letters, of which 
    \begin{align*}
    F_{n,m-(n-1)} &\mbox{ are } a_n\mbox{'s},\\
    F_{n,m-n} &\mbox{ are } a_1\mbox{'s},\\
    F_{n,m-(n+1)} &\mbox{ are } a_2\mbox{'s},\\
    \vdots\\
    F_{n,m-(2n-2)} &\mbox{ are } a_{n-1}\mbox{'s}.\\
    \end{align*}
    \item [(2)] For any $m\in\mathbb{N}_{\ge 1}$, the concatenation
    $$S_{n+(n-1)m}\circ \cdots\circ S_{n+(n-1)}\circ S_n$$ gives the first $F_{n,n}+F_{n,2n-1}+\cdots+F_{n,(m+1)n-m}$ letters of $S_\infty$. 
\end{itemize}
\end{thm}
\begin{thm}\label{m5}
Let $n\ge 3$ and $m\ge 1$. Let $F_{n,c_1}+F_{n,c_2}+\cdots+F_{n,c_k}$ be the $n$-decomposition of $m\in \mathbb{N}$. Then $S_{c_k}\circ S_{c_{k-1}}\circ \cdots\circ S_{c_1}$ gives the first $m$ letters of $S_\infty$.
\end{thm}
\begin{rek}\label{left}\normalfont
In order that Theorem \ref{m1} item (1) makes sense, we need to extend the sequence $F_{n,m}$ to the left while following the recursive relation. It is straightforward that for $n\ge 3$, we have $F_{n,0}=\cdots = F_{n,2-n}= 0$, $F_{n,1-n} = 1$, and $F_{-n} = \cdots = F_{3-2n} = 0$. 
\end{rek}
For each $m\ge 1$, let $N_{a_i}(m)$ denote the number of $a_i$ in the string $S_\infty$ up to $m$. 
\begin{thm}
Let $n\ge 3$ and $m\ge 1$. If $m = F_{n,c_1}+F_{n,c_2} + \cdots+F_{n,c_k}$ is an $n$-decomposition of $m$, then 
$$ N_{a_i}(m) \ =\ \begin{cases} F_{n, c_1-(n+i-1)} + F_{n, c_2-(n+i-1)} + \cdots + F_{n, c_k-(n+i-1)}, &\mbox{ if } 1\le i\le n-1;\\
    F_{n, c_1-(n-1)} + F_{n, c_2-(n-1)} + \cdots + F_{n, c_k-(n-1)}, &\mbox{ if } i = n.
\end{cases}$$
\end{thm}

Our final result extends \cite[Theorem 3.4]{MG1}, which describes the set of all positive integers having the summand $F_{k}$ in their Zeckendorf decomposition. The following theorem sheds another light on the relationship between the $n$-string and the $n$-decomposition. 

\begin{thm}\label{m2}
For $k\ge n\ge 3$, the set of all positive integers having the summand $F_{n,k}$ in their $n$-decomposition is given by 
\begin{align*}
    Z_{n}(k) \ =\ \left\{j + F_{n,k}+\sum_{i=1}^n F_{n, k+i}\cdot N_{a_i}(m): 0\le j\le F_{n,k-(n-1)}-1 \mbox{ and }m\ge 0\right\}.
\end{align*}
\end{thm}

\begin{rek}\normalfont
In \cite[Theorem 3.4]{MG1}, $Z_{2}(k)$ has a closed form thanks to \cite[Theorem 3.3]{MG2}, which provides a neat formula for $N_{a_i}(m)$ in the case of the golden string. The formula was deduced using the Binet's formula for the Fibonacci numbers. However, the author of the present paper is unable to find such a closed form for $Z_{n}(k)$ when $n\ge 3$. Thus, Theorem \ref{m2} only gives another (not quicker) way to find $Z_n(k)$ and shows a relationship between the $n$-string and the $n$-decomposition. 
\end{rek}

As we proceed to the proof of the main theorems, a number of interesting immediate results are encountered. 
\section{Relationship between the $n$-decomposition and the $n$-string}
The following lemma will be used in due course. 
\begin{lem}\label{l1} Let $n\ge 3$ and $S_\infty$ be the $n$-string. The following items hold.
\begin{itemize}
\item[(1)] Fix $i$ and $j$ such that $n\le i\le j-(n-1)$. Then $S_j\circ S_i$ gives the first $F_{n,i}+F_{n,j}$ letters of $S_\infty$.
\item[(2)] Fix $j\ge i\ge n$. There exists $S^*$ (possibly empty), a concatenation of some $S_i$'s, such that $S_i\circ S^* = S_j$.
\item[(3)] For $m\ge 2$, there exists $S^*$ (possibly empty), a concatenation of some $S_i$'s, such that 
$$(S_{n+(n-1)m}\circ \cdots \circ S_{n+(n-1)}\circ S_n)\circ S^* \ =\ S_{n+(n-1)m+1}\circ S_i$$
for some $n\le i\le n+(n-1)(m-1)+1$.
\item [(4)] For $m\ge 2$, there exists $S^*$ (possibly empty), a concatenation of some $S_i$'s, such that 
$$(S_{n+(n-1)m}\circ \cdots \circ S_{n+(n-1)}\circ S_n)\circ S^* \ =\ S_{n+(n-1)m+2}.$$
\end{itemize}
\end{lem}
\begin{proof}
We first prove item (1). By construction, we have $S_{j+1} = S_j\circ S_{j-(n-1)}$ and $S_{j+1}$ gives the first $F_{n,j+1}$ letters of $S_{\infty}$. Hence, $S_j\circ S_{j-(n-1)}$ gives the first $F_{n, j+1}$ letters of $S_{\infty}$. Because $n\le i\le j-(n-1)$, $S_i$ gives the first $F_{n, i}$ letters of $S_{j-(n-1)}$. Therefore, $S_{j}\circ S_{i}$ gives the first $F_{n,i}+F_{n,j}$ letters of $S_\infty$, as desired. 

To prove item (2), it suffices to show that there exists $S^*$ such that $S_i\circ S^* = S_{i+1}$. By construction and the fact that $i\ge n$, we can let $S^* = S_{i-(n-1)}$. 

Next, we prove item (3). We proceed by induction on $m$. \textit{Base cases:} for $m = 2$, we have
$$S_{n+2(n-1)}\circ S_{n+(n-1)}\circ S_n \ =\ S_{n+2(n-1)+1}\circ S_n.$$
Thus, letting $S^*$ be the empty string, we have our claim hold. For $m = 3$, we have 
$$S_{n+3(n-1)}\circ S_{n+2(n-1)}\circ S_{n+(n-1)}\circ S_n\ =\ S_{n+3(n-1)+1}\circ S_{n+(n-1)+1}.$$
Again, letting $S^*$ be the empty string, we have our claim hold. \textit{Inductive hypothesis:} suppose that our claim holds for all $2\le m\le \ell$ for some $\ell\ge 3$. We have
\begin{align*}
    &S_{n+(n-1)(\ell+1)}\circ S_{n+(n-1)\ell}\circ  S_{n+(n-1)(\ell-1)}\circ \cdots \circ S_{n+(n-1)}\circ S_n\\
    &\ =\ S_{n+(n-1)(\ell+1)+1}\circ (S_{n+(n-1)(\ell-1)}\circ \cdots \circ S_{n+(n-1)}\circ S_n)\\
    &\ =\ S_{n+(n-1)(\ell+1)+1}\circ S_{n+(n-1)(\ell-1)+1}\circ S_i \mbox{ for some } n\le i\le n+(n-1)(\ell-2)+1.
\end{align*}
The last equality is due to the inductive hypothesis. By item (2), there exists some $S'$, a concatenation of some $S_i$'s, such that $S_i\circ S' = S_{n+(n-1)(\ell-2)+1}$. Hence, 
\begin{align*}
    &(S_{n+(n-1)(\ell+1)}\circ S_{n+(n-1)\ell}\circ  S_{n+(n-1)(\ell-1)}\circ \cdots \circ S_{n+(n-1)}\circ S_n)\circ S'\\
    &\ =\ S_{n+(n-1)(\ell+1)+1}\circ S_{n+(n-1)(\ell-1)+1}\circ S_i\circ S'\\
    &\ =\ S_{n+(n-1)(\ell+1)+1}\circ S_{n+(n-1)(\ell-1)+1}\circ S_{n+(n-1)(\ell-2)+1}\\
    &\ =\ S_{n+(n-1)(\ell+1)+1}\circ S_{n+(n-1)(\ell-1)+2}.
\end{align*}
Therefore, our claim holds for $m = \ell+1$. This completes our proof of item (3).

Finally, item (4) follows immediately from items (2) and (3). 
\end{proof}
We are ready to prove Theorems \ref{m1} and \ref{m5}.
\begin{proof}[Proof of Theorem \ref{m1}]
First, we prove item (1). The first part is clear by construction. We prove the second part by induction on $m$. \textit{Base cases}: if $1\le m\le n-1$, then $S_m = a_m$ by construction. All numbers in the set $\{m-(n-1), m-n, \ldots, m-(2n-2)\}$ lie between $3-2n$ and $0$ inclusive. By Remark \ref{left}, for all $3-2n\le i\le 0$ except $i= 1-n$, we have $F_{n,i}  = 0$, while $F_{n, 1-n} = 1$. Write $F_{n, 1-n} = F_{n, m-(m+(n-1))}$ to see that our claim holds. For $m = n$, it is easy to check that our claim also holds. 

\textit{Inductive hypothesis:} suppose that our claim holds for all $1\le m\le \ell$ for some $\ell \ge n$. Because $S_{\ell+1} = S_{\ell}\circ S_{\ell+1-n}$, the number of $a_n$'s in $S_{\ell+1}$ is 
\begin{align*}&F_{n, \ell-(n-1)} + F_{n, \ell+1-n-(n-1)} \mbox{ by the inductive hypothesis},\\
&\ =\ F_{n, \ell-n+2}\ =\ F_{n, (\ell+1)-(n-1)}
\end{align*}
Similarly, the formulas for the number of $a_i$'s in $S_{\ell+1}$ are all correct. This completes our proof of item (1). 
\end{proof}

Next, we prove item (2). For $m = 1$, we have $S_{n+(n-1)}\circ S_n = S_{2n}$, which gives the first $F_{n,2n} = F_{n, n} + F_{n, 2n-1}$ letters of $S_{\infty}$. For $m \ge 2$, item (2) follows from the first part of item (1) and Lemma \ref{l1} item (4).  

\begin{proof}[Proof of Theorem \ref{m5}]
We prove by induction on $k$, the number of terms in the $n$-decomposition. \textit{Base case:} If $k = 1$, the statement of the lemma is certainly true. \textit{Inductive hypothesis:} assume that the statement is true for some $k = \ell\ge 1$. Consider the $n$-decomposition
$$m \ =\ F_{n,c_1} + F_{n,c_2} + \cdots + F_{n,c_\ell} + F_{n,c_{\ell+1}}.$$
By the inductive hypothesis, we know that
$$S_{c_\ell}\circ S_{c_{\ell-1}}\circ \cdots \circ S_{c_1}$$
gives the first $F_{n,c_1}+F_{n,c_2}+\cdots+F_{n,c_\ell}$ letters of $S_\infty$, and therefore, by Theorem \ref{m1} item (2), of 
$$S_{c_{\ell+1}-1}\circ S_{c_{\ell+1}-2}\circ \cdots \circ S_{c_{\ell+1}-\ell}.$$
Hence, we know that 
$$S_{c_{\ell+1}}\circ S_{c_\ell}\circ S_{c_{\ell-1}}\circ \cdots \circ S_{c_1}$$
gives the first $m$ letters of 
$$S_{\ell+1}\circ S_{c_{\ell+1}-1}\circ S_{c_{\ell+1}-2}\circ \cdots \circ S_{c_{\ell+1}-\ell},$$
and hence, by Theorem \ref{m1} item (2), of $S_\infty$. 
\end{proof}
\section{Fixed term in the $n$-decomposition}
In this section, we fix $k\ge n\ge 3$. Our goal is to characterize the set $Z_n(k)$, the set of all positive integers having the summand $F_{n,k}$ in their $n$-decomposition.
The following lemma generalizes \cite[Lemma 2.5]{MG1}.
\begin{lem}\label{l2}
Let $m\ge 1$. For $v\ge 1$ and $1\le u\le n$, we have 
\begin{align}F_{n, m+vn+u} - (F_{n, m+vn+(u-1)}+\cdots+F_{n, m+n+(u-1)}) \ =\ F_{n, m+u}.\end{align}
\end{lem}
\begin{proof} If $v = 1$, the identity holds due to the linear recurrence relation of $\{F_{n,i}\}$. Assume that $v\ge 2$. 
We have
\begin{align*}
    &(F_{n, m+vn+(u-1)}+\cdots+F_{n, m+n+(u-1)}) + F_{n, m+u}\\
    &\ =\ (F_{n, m+vn+(u-1)}+\cdots+F_{n, m+2n+(u-1)}) + (F_{n, m+n+(u-1)} + F_{n, m+u})\\
    &\ =\ (F_{n, m+vn+(u-1)}+\cdots+F_{n, m+2n+(u-1)}) + F_{n, m+n+u}\\
    &\ \ \ \vdots\\
    &\ =\ F_{n, m+vn+(u-1)} + F_{n, m+(v-1)n+u}\ =\ F_{n, m+vn+u}. 
\end{align*}
This completes our proof. 
\end{proof}
\begin{lem}\label{l3}
Let $i$ and $j\in\mathbb{N}$. If $F_{n,u}$ and $F_{n, v}$ are the largest summands in the $n$-decompositions of $i$ and $j$, respectively, then $u>v$ implies that $i>j$. 
\end{lem}
The statement of the lemma is obvious since the $n$-decomposition of a number can be found by the greedy algorithm. 
\begin{lem}\label{l20}
For $j\ge 1$, the $(F_{n,nj})$th character of $S_\infty$ is $a_n$; the $(F_{n,nj+1})$th character of $S_\infty$ is $a_1$; $\ldots$; $(F_{n,nj+(n-1)})$th character of $S_\infty$ is $a_{n-1}$.
\end{lem}
\begin{proof}
\textit{Base case:} By construction, we know that the statement is true for $j = 1$. 

\textit{Inductive hypothesis:} suppose the statement is true for $j = m$ for some $m\ge 1$. The $(F_{n, n(m+1)})$th of $S_\infty$ is the last letter of $S_{n(m+1)} = S_{n(m+1)-1}\circ S_{nm}$. By the inductive hypothesis, the last letter of $S_{nm}$ is $a_n$ and so, the $(F_{n, n(m+1)})$th letter of $S_\infty$ is $a_n$. The proof is completed by similar arguments for the $(F_{n, n(m+1)+u})$th letter, where $1\le u\le n-1$. 
\end{proof}

Let $\mathcal{X}_{n,k}$ denote the set of all positive integers whose $n$-decompositions have $F_{n,k}$ as the smallest summand. Next, let $\mathcal{Q}_{n,k} = \{q(j)\}_{j\ge 1}$ be the strictly increasing infinite sequence that results from arranging the elements of $\mathcal{X}_{n,k}$ into ascending numerical order. Table 1 gives the ordered list of summands for each $q(j)$, where the $r$th row corresponds to the $r$th smallest element from $\mathcal{X}_{n,k}$. (Lemma \ref{l3} helps explain the ordering of rows.)

\begin{tabular}{cccccccc}
 Row  &  & &  &  &         & \\
 \hline
1& & $F_{n,k}$ &  &  &  &         &     \\
2& & $F_{n,k}$ & $F_{n, k+n}$ &  &  &     &    \\
3& &$F_{n,k}$ &  & $F_{n,k+n+1}$ &  &    &     \\
  & & &  \vdots &  &         & \\
$(n+1)$& &$F_{n,k}$ &  &   & $F_{n,k+2n-1}$  &   &     \\
$(n+2)$& & $F_{n,k}$ &  &  &  & $F_{n, k+2n}$  &    \\
$(n+3)$& &$F_{n,k}$ &  $F_{n, k+n}$ &  &  &  $F_{n, k+2n}$ &   \\
$(n+4)$& &$F_{n,k}$ &  &  &  &    &  $F_{n, k+2n+1}$   \\
$(n+5)$& &$F_{n,k}$ & $F_{n, k+n}$  &  &  &    &  $F_{n, k+2n+1}$   \\
$(n+6)$& &$F_{n,k}$ &  & $F_{n, k+n+1}$ &  &    &  $F_{n, k+2n+1}$\\
  &  & & \vdots &  &         & 
\end{tabular}

\begin{center}
Table 1. The $n$-decompositions, in numerical order, of the positive integers having $F_{n,k}$ as their smallest summand.  
\end{center}

\begin{lem}\label{k1}
For $j\ge n$, the rows of Table 1 for which $F_{n,k+j}$ is the largest summand are those numbered from $F_{n,j}+1$ to $F_{n,j+1}$ inclusive. 
\end{lem}
\begin{proof}
We prove by induction on $j$. \textit{Base cases:} for $n\le j\le 2n-1$, there is exactly one row with $F_{n,k+j}$ being the largest summand. In particular, the row with $F_{n,k+j}$ being the largest summand is the $(j-(n-2))th$. Hence, the base cases are done if we prove the two following claims:
\begin{itemize}
    \item [(1)] $F_{n, j+1} - F_{n, j} = 1$;
    \item [(2)] $F_{n,j} + 1 = j-(n-2)$.
\end{itemize}
For the first claim, we have $F_{n, j+1} - F_{n,j} = F_{n, j-(n-1)} = 1$ for all $n\le j\le 2n-1$. This is due to the construction of $\{F_{n, j}\}$. To prove the second claim, observe that as $j$ goes from $n$ to $2n-1$, we have $F_{n,j}$ increase from $1$ to $n$.

\textit{Inductive hypothesis:} suppose that the statement of the lemma holds for $n\le j\le m$ for some $m\ge 2n-1$. We want to show that the rows with $F_{n, k+m+1}$ being the largest summand is from $F_{n, m+1}+1$ to $F_{n, m+2}$. By the inductive hypothesis, the number of rows with the largest summand not larger than $F_{n, k+m+1-n}$ is 
\begin{align*}
    1 + \sum_{j=n}^{m+1-n}(F_{n,j+1}-F_{n,j}) \ =\ 1+ F_{n, m+2-n} - F_{n,n} \ =\ F_{n, m+2-n}, 
\end{align*}
which is also the number of rows with $F_{n, k+m+1}$ being the largest summand. 

By the inductive hypothesis, the rows with $F_{k+m}$ being the largest is from $F_{n,m}+1$ to $F_{n,m+1}$. Therefore, the rows with $F_{n, k+m+1}$ being the largest is from $F_{n,m+1}+1$ to $F_{n,m+1}+F_{n,m+2-n} = F_{n,m+2}$.
\end{proof}

\begin{lem}\label{k100}
For $j\ge 1$, we have
\begin{align}\label{k10}q(j+1) - q(j) \ =\ \begin{cases}F_{n,k+1}, &\mbox{ if the }j\mbox{th character of }S_\infty \mbox{ is } a_1;\\
F_{n,k+2}, &\mbox{ if the }j\mbox{th character of }S_\infty \mbox{ is } a_2;\\
\mbox{ }\mbox{ }\mbox{ }\vdots\\
F_{n,k+n}, &\mbox{ if the }j\mbox{th character of }S_\infty \mbox{ is } a_n.
\end{cases}\end{align}
\end{lem}
\begin{proof}
We prove by induction on $j$. \textit{Base cases:} consider $1\le j\le F_{n, 2n}-1 = n$. From Table 1, we know that 
\begin{align}\label{l10}q(j+1) - q(j) \ =\  \begin{cases} F_{n,k+n}, &\mbox{ if }j = 1;\\
F_{n,k+(j-1)}, &\mbox{ if } 2\le j\le n.
\end{cases}\end{align}
Using the fact that $S_\infty = a_na_1a_2\ldots a_{n-1}\ldots$ and \eqref{l10}, we have proved the base cases for $1\le j\le F_{n, 2n}-1$.

\textit{Inductive hypothesis:} Assume that the statement is true for $1\le j\le F_{n,m} - 1$ for some $m\ge 2n$. By Lemma \ref{k1}, the first $F_{n,m-n+1}$ rows of Table 1 are those for which the largest summand is no greater than $F_{n, k+m-n}$, and the rows for which $F_{n, k+m}$ is the largest summand are from $F_{n,m}+1$ to $F_{n, m+1}$. 

Due to the ordering of rows in Table 1, we have $q(i) + F_{n,k+m} = q(i+F_{n,m})$ for $1\le i\le F_{n, m-n+1}$. Hence, for $1\le i\le F_{n,m-n+1}-1$, we have 
\begin{align*}
    q(i+1+F_{n,m}) - q(i+F_{n,m}) &\ =\ (q(i+1)+F_{n,k+m}) - (q(i)+F_{n,k+m})\\
    &\ =\ q(i+1) - q(i).
\end{align*}
By the construction of $S_\infty$, the string of the first $F_{n,m-n+1}$ characters is identical to the string of characters from the $(F_{n,m}+1)$th to the $F_{n,m+1}$th inclusive. Therefore, we know that \eqref{k10} is true for $1+F_{n,m}\le j\le F_{n,m+1}-1$. It remains to show that \eqref{k10} is true for $j = F_{n,m}$. We have
\begin{align*}
    q(F_{n,m}+1) - q(F_{n,m}) \ =\ F_{n, k+m} - (F_{n,k+m-1}+F_{n, k+m-1-n}+\cdots + F_{n,k+m-1-\ell n}),
\end{align*}
where $\ell$ satisfies $n\le m-1-\ell n < 2n$. Write $m = vn + u$ for some $1\le u\le n$. It follows that $1-(u-1)/n\le v-\ell < 2-(u-1)/n$, so $v-\ell = 1$. Hence, 
\begin{align*}
    &q(F_{n,m}+1) - q(F_{n,m})\\
    &\ =\ F_{n, k+vn+u} - (F_{n,k+vn+u-1}+F_{n, k+(v-1)n+u-1}+\cdots + F_{n,k+n+u-1}) \ =\ F_{n, k+u}
\end{align*}
due to Lemma \ref{l2}. Using Lemma \ref{l20}, we conclude that \eqref{k10} is true for $j = F_{n,m}$. This completes our proof. 
\end{proof}
Finally, we prove Theorem \ref{m2}.
\begin{proof}[Proof of Theorem \ref{m2}]
By Lemma \ref{k100}, we can write
$$\mathcal{X}_{n,k} \ =\ \left\{F_{n,k} + \sum_{i=1}^n F_{n, k+i}\cdot N_{a_i}(m): m \ge 0\right \}.$$
The numbers in $\{F_{n,n}, F_{n, n+1}, \ldots, F_{n, k-n}\}$ are used to obtain the $n$-decompositions of all integers for which the largest summand is no greater than $F_{n, k-n}$. In particular, such $n$-decompositions generate all integers from $1$ to $F_{n, k-n+1}-1$ inclusive. Furthermore, such decompositions can be appended to any $n$-decomposition having $F_{n,k}$ as its smallest summand to produce another $n$-decomposition. Therefore, we know that
\begin{align*}
    Z_{n}(k) \ =\ \left\{j + F_{n,k}+\sum_{i=1}^n F_{n, k+i}\cdot N_{a_i}(m): 0\le j\le F_{n,k-(n-1)}-1 \mbox{ and }m\ge 0\right\},
\end{align*}
as desired. 
\end{proof}

\ \\
\end{document}